\documentclass[9pt, a4paper]{amsart}

\usepackage{xcolor}
\definecolor{webgreen}{rgb}{0,.5,0}
\definecolor{webbrown}{rgb}{.6,0,0}
\definecolor{RoyalBlue}{cmyk}{1, 0.50, 0, 0}
\usepackage[colorlinks=true, breaklinks=true, urlcolor=webbrown, linkcolor=RoyalBlue, citecolor=webgreen,backref=page]{hyperref}

\usepackage{epsfig, graphicx, subfigure}
\usepackage{verbatim, setspace}
\usepackage{amsmath, amssymb, bm}

\usepackage{mathptmx}

\newtheorem{theorem}{Theorem}

\newtheorem{lemma}[theorem]{Lemma}
\theoremstyle{remark}

\newcommand{\qandq}{\quad \text{and} \quad}

\begin{document}

\title{On an identity by Ercolani,  Lega, and Tippings}

\author{Maxim L. Yattselev}

\address{Department of Mathematical Sciences, Indiana University Indianapolis 
}

\email{\href{mailto:maxyatts@iu.edu}{maxyatts@iu.edu}}

\thanks{The research was supported in part by a grant from the Simons Foundation, CGM-706591.}



\begin{abstract}
In this note we prove that
\[
j!\,2^N \, \binom{N+j-1}{j} \, {}_2F_1\left(\begin{matrix}-j,-2j \\ -N-j+1 \end{matrix};-1\right) = \sum_{l=0}^N \binom{N}{l}\prod_{i=0}^{j-1}2(2i+1+l),
\]
where \( N \) and \( j \)  are positive integers, which resolves a question posed by Ercolani,  Lega, and Tippings.  
\end{abstract}

\maketitle


In \cite[Theorem~2.1]{ErcLegaTip23}, Ercolani, Lega, and Tippings have shown that the number of \( 2\nu \)-valent maps with \( j \) vertices and two legs that can be embedded in a surface of genus \( g\geq 1 \) is given by
\[
j! \left[2\nu(\nu-1)\binom{2\nu-1}{\nu-1}\right]^j \sum_{l=0}^{3g-1} a_l(g,\nu)  \binom{2g-2+l+j}{j} {}_2F_1\left(\begin{matrix}-j,-\nu j \\  2-2g-l-j \end{matrix} \, ;\frac{1}{1-\nu}\right)
\]
for some coefficients \( a_l(g,\nu) \). Based on numerical computations, they subsequently conjectured, see \cite[Conjecture~4.2]{ErcLegaTip23} or \cite[Conjecture~5.1]{ErcLegaTip}, that when \( \nu=2 \) the summands in the above expression can be stated without hypergeometric functions with the help of  identity \eqref{identity} further below. In this note we provide a proof of this fact.

\begin{theorem}
Let \( N,j \) be positive integers. Then
\begin{equation}
\label{identity}
j!2^N \binom{N+j-1}{j} {}_2F_1\left(\begin{matrix}-j,-2j \\ -N-j+1 \end{matrix};-1\right) = \sum_{l=0}^N \binom{N}{l}\prod_{i=0}^{j-1}2(2i+1+l).
\end{equation}
\end{theorem}

We use the following notation for the falling and rising factorials:
\[
(a)^{\underline 0} = (a)^{\overline 0} : =1, \quad (a)^{\underline n} := a(a-1)\cdots(a-n+1), \qandq (a)^{\overline n} := a(a+1)\cdots(a+n-1)
\]
for \( n\geq 1 \). We prove \eqref{identity} in three steps formulated as separate lemmas.

\begin{lemma}
The right-hand side of \eqref{identity} is equal to \( 2^N R_j(N) \), where \( R_j(x) = \sum_{i=0}^j R(i,j) (x)^{\underline i} \) and \( R(i,j) \) are numbers uniquely determined by the relations
\[
R(0,j) = 2^j (2j-1)!!, \quad R(j,j) = 1, \quad R(i,j+1) = 2(2j+i+1) R(i,j) + R(i-1,j),
\]
with the recurrence relation holding for \( i\in\{1,2,\ldots,j\} \).
\end{lemma}
\begin{proof}
Let \( C_{i,j} \) be the following coefficients:
\[
\prod_{i=0}^{j-1}(2i+1+x) = \sum_{k=0}^j C_{k,j} x^k.
\]
Since
\[
\prod_{i=0}^j(2i+1+x) = (2j+1+x) \sum_{k=0}^j C_{k,j} x^k = \sum_{k=0}^j (2j+1) C_{k,j} x^k + \sum_{k=1}^{j+1} C_{k-1,j} x^k,
\]
it holds that
\[
C_{0,j} = (2j-1)!!, \quad C_{j,j} = 1, \quad C_{k,j+1} = (2j+1)C_{k,j} + C_{k-1,j}, \;\; k\in\{1,2,\ldots,j\}.
\]
Exactly as in the case of the coefficients \( R(i,j) \), the coefficients \( C_{k,j} \) are uniquely defined by the above relations since the knowledge of all the coefficients on the level \( j \) allows one to compute all the coefficients on the level \( j+1 \) with the base case \( C_{0,1} = C_{1,1} = 1 \). Recall \cite[Equation~(26.8.10)]{DLMF} that
\[
x^k = \sum_{i=1}^k \begin{Bmatrix} k \\ i \end{Bmatrix} (x)^{\underline i},
\]
where \( \begin{Bmatrix} k \\ i \end{Bmatrix} \) are Stirling numbers of the second kind.  Therefore,
\[
\prod_{i=0}^{j-1}(2i+1+x) = C_{0,j} + \sum_{i=1}^j \left(\sum_{k=i}^j C_{k,j}\begin{Bmatrix} k \\ i \end{Bmatrix}\right) (x)^{\underline i}.
\]
Observe that
\begin{align*}
\sum_{l=0}^N \binom Nl (l)^{\underline i} & = \sum_{l=0}^N \frac{N!}{l!(N-l)!}(l)(l-1)\cdots (l-i+1) \\
& = \sum_{l=i}^N \frac{N!}{(l-i)!(N-l)!} =  \sum_{l=0}^{N-i} \frac{N!}{l!(N-i-l)!} = 2^{N-i} (N)^{\underline i}.
\end{align*}
Therefore, the right-hand side of \eqref{identity} is equal to
\[
2^j \sum_{l=0}^N\binom Nl \left(C_{0,j} + \sum_{i=1}^j \left(\sum_{k=i}^j C_{k,j}\begin{Bmatrix} k \\ i \end{Bmatrix}\right) (l)^{\underline i} \right)  = 2^N R_j(N),
\]
where \( R_j(x) := \sum_{i=0}^j R(i,j) (x)^{\underline i} \) with
\[
R(0,j) := 2^j(2j-1)!! \qandq R(i,j) := 2^{j-i}\sum_{k=i}^j C_{k,j}\begin{Bmatrix} k \\ i \end{Bmatrix}, \;\; i\in\{1,2,\ldots,j\}.
\]
Since \( \begin{Bmatrix} j \\ j \end{Bmatrix}= 1 \), see  \cite[Equation~(26.8.4)]{DLMF}, it indeed holds that \( R(j,j) = 1 \). Thus, we only need to establish the recurrence relation. The recurrence relation for \( C_{k,j} \) yields that
\begin{align*}
R(i,j+1) & = 2^{j+1-i}\sum_{k=i}^{j+1} C_{k,j+1} \begin{Bmatrix} k \\ i \end{Bmatrix} = 2^{j+1-i}\sum_{k=i}^j C_{k,j+1}\begin{Bmatrix} k \\ i \end{Bmatrix} + 2^{j+1-i}\begin{Bmatrix} j+1 \\ i \end{Bmatrix} \\
& = 2(2j+1) R(i,j) + 2^{j+1-i} \sum_{k=i}^{j+1} C_{k-1,j}\begin{Bmatrix} k \\ i \end{Bmatrix}
\end{align*}
for any \( i\in\{1,2,\ldots, j\} \), where we also used the fact that \( C_{j,j}=1 \). Furthermore, we get from \cite[Equation~(26.8.22)]{DLMF} that
\begin{align*}
2^{j+1-i} \sum_{k=i}^{j+1} C_{k-1,j}\begin{Bmatrix} k \\ i \end{Bmatrix} & = 2^{j+1-i} \sum_{k=i-1}^j C_{k,j} \begin{Bmatrix} k+1 \\ i \end{Bmatrix} = 2^{j+1-i} \sum_{k=i-1}^j C_{k,j} \left( i\begin{Bmatrix} k \\ i \end{Bmatrix} + \begin{Bmatrix} k \\ i-1 \end{Bmatrix}\right) \\
&  = 2i R(i,j) + R(i-1,j),
\end{align*}
where, by convention, \( \begin{Bmatrix} i-1 \\ i \end{Bmatrix}= 0\). This finishes the proof of the lemma.
\end{proof}

\begin{lemma}
The left-hand side of \eqref{identity} is equal to \( 2^N L_j(N) \), where \( L_j(x) = \sum_{i=0}^j L(i,j) (x)^{\underline i} \) with
\[
L(0,j ) = \frac{(2j)!}{j!} \qandq L(i,j) = \frac{j!}{i!}\sum_{k=i}^j \binom{2j}{j+k} \binom{k-1}{i-1}, \quad i\in\{1,2,\ldots,j\}.
\]
\end{lemma}
\begin{proof}
It follows from \cite[Equation~(15.2.1)]{DLMF} that the left-hand side of \eqref{identity} is equal to \( 2^N L_j(N) \), where
\[
L_j(x) := (x)^{\overline j}\sum_{k=0}^j \frac{(-j)^{\overline k}(-2j)^{\overline k}}{(-x-j+1)^{\overline k}}\frac{(-1)^k}{k!}
\]
(we take this formula as the definition of \( L_j(x) \) and show that this polynomial also admits an expression as in the statement of the lemma). Since
\[
(-a)^{\overline k} = (-1)^k(a)^{\underline k} \qandq \frac{(a)^{\overline j}}{(a+j-1)^{\underline k}} = (a)^{\overline{j-k}},
\]
we have that
\begin{align*}
L_j(x) & = \sum_{k=0}^j (j)^{\underline k}(2j)^{\underline k} \frac{(x)^{\overline{j-k}}}{k!} = \sum_{k=0}^j \frac{j!}{(j-k)!}\frac{(2j)!}{(2j-k)!} \frac{(x)^{\overline{j-k}}}{k!} \\
& = \sum_{k=0}^j \binom{j}{k} \frac{(2j)!}{(2j-k)!} (x)^{\overline{j-k}} = \sum_{k=0}^j \binom{j}{k} \frac{(2j)!}{(j+k)!} (x)^{\overline k},
\end{align*}
where we replaced \( j-k \) by \( k \) to get the last equality. It is known \cite{Lah55}  that
\[
(x)^{\overline k} = \sum_{i=1}^{k} \binom{k-1}{i-1}\frac{k!}{i!} (x)^{\underline i}, \quad k\geq 1.
\]
Hence,
\begin{align*}
L_j(x) &  = \frac{(2j)!}{j!} + \sum_{k=1}^j \binom{j}{k} \frac{(2j)!}{(j+k)!} \sum_{i=1}^{k} \binom{k-1}{i-1}\frac{k!}{i!} (x)^{\underline i} \\
& = \frac{(2j)!}{j!} + \sum_{i=1}^j \left( \sum_{k=i}^j \binom{j}{k} \frac{(2j)!}{(j+k)!}\binom{k-1}{i-1}\frac{k!}{i!} \right) (x)^{\underline i},
\end{align*}
which finishes the proof of the lemma.
\end{proof}

\begin{lemma}
It holds that \( L(i,j) = R(i,j) \) for all \(  i\in\{0,1,\ldots,j\} \) and \( j\geq 1 \). In particular, \eqref{identity} is true.
\end{lemma}
\begin{proof}
Clearly, \( L(0,j) = R(0,j) \) and \( L(j,j) = R(j,j) = 1\). Thus, we only need to show that
\[
L(i,j+1) = 2(2j+i+1) L(i,j) + L(i-1,j), \quad i\in\{1,2,\ldots,j\},
\] 
since this recurrence relation and the marginals \( L(0,j) \), \( L(j,j) \) uniquely determine the whole table \( L(i,j) \). In what follows, we agree that binomial coefficients with out-of-range indices are set to zero. In what follows we repeatedly use the elementary identity
\[
\binom{n+1}{l} = \binom nl + \binom n{l-1}.
\]
Using this identity twice and our convention concerning binomial coefficients with out-of-range indices, we get that
\begin{align*}
L(i,j+1) &= 2(j+1)L(i,j) + \frac{(j+1)!} {i!} \sum_{k=i}^{j+1}\left[ \binom{2j}{j+k+1} + \binom{2j}{j+k-1} \right]\binom{k-1}{i-1} \\
& = 2(j+1)L(i,j) + \frac{(j+1)!} {i!} \left[\sum_{k=i+1}^{j} \binom{2j}{j+k}\binom{k-2}{i-1} + \sum_{k=i-1}^j \binom{2j}{j+k}\binom{k}{i-1}\right],
\end{align*}
where the second row is obtained simply by a change of summation indices. Therefore, when \( i=1 \), we get that
\begin{align*}
L(1,j+1) &= 4(j+1)L(1,j) + (j+1)! \left[ \binom{2j}{j} - \binom{2j}{j+1} \right] \\
& =4(j+1)L(1,j) + L(0,j) \big[j+1 - j \big],
\end{align*}
which establishes the desired recurrence relation for \( i=1 \). On the other hand, when \( i>1 \), the above sum in square brackets can further be rewritten with the help of the elementary identity as
\begin{multline*}
\sum_{k=i+1}^j \binom{2j}{j+k}\left[ \binom{k-1}{i-1} - \binom{k-2}{i-2} \right]  + \sum_{k=i-1}^j \binom{2j}{j+k}\left[ \binom{k-1}{i-1} + \binom{k-1}{i-2} \right]  = \\  2\sum_{k=i}^j \binom{2j}{j+k} \binom{k-1}{i-1} + \sum_{k=i-1}^j \binom{2j}{j+k}\left[ \binom{k-1}{i-2} -  \binom{k-2}{i-2}  \right],
\end{multline*}
where the last sum simply reduced to \( \binom{2j}{j+1} \) when \( i=2 \). Thus, using the elementary identity once more, we get that
\[
L(i,j+1) = 4(j+1)L(i,j) + \frac{(j+1)!} {i!} \sum_{k=i-1}^j \binom{2j}{j+k} \binom{k-2}{i-3},
\]
which holds even for \( i=2 \) if we understand that \( \binom{-1}{-1} = 1 \) and \( \binom{k-2}{-1} = 0 \), \( k\geq 2 \). That is, to prove the lemma we need to show that
\[
2(i-1) L(i,j) + L(i-1,j) = \frac{(j+1)!} {i!}  \sum_{k=i-1}^j \binom{2j}{j+k} \binom{k-2}{i-3},
\]
or equivalently that
\[
\sum_{k=i-1}^j \binom{2j}{j+k} \left [ 2(i-1) \binom{k-1}{i-1} + i \binom{k-1}{i-2} -  (j+1) \binom{k-2}{i-3} \right] = 0.
\] 
When \( i=2 \), the left-hand side above becomes
\[
\sum_{k=2}^j 2k\binom{2j}{j+k} - (j-1) \binom{2j}{j+1} = \sum_{k=2}^j (j+k)\binom{2j}{j+k} -  \sum_{k=1}^{j-1} (j-k)\binom{2j}{j+k}.
\]
Similarly, when \( i>2 \), the left-hand side in question is equal to
\begin{multline*}
\sum_{k=i-1}^j \binom{2j}{j+k} \frac{(k-2)!}{(k-i+1)!(i-2)!} \big [  2(k-1)(k-i+1) + i(k-1) - (j+1)(i-2) \big] = \\
\sum_{k=i-1}^j \binom{2j}{j+k} \frac{(k-2)!}{(k-i+1)!(i-2)!} \big [  (k-i+1)(j+k) - (j-k)(k-1) \big],
\end{multline*}
which is the same as
\[
\sum_{k=i}^j  (j+k) \binom{2j}{j+k} \binom{k-2}{i-2} -  \sum_{k=i-1}^{j-1}  (j-k) \binom{2j}{j+k} \binom{k-1}{i-2}
\]
(as shown just before, this formula is valid for \( i=2 \) as well). This difference is indeed equal to zero as the second sum above can be rewritten as
\[
\sum_{k=i-1}^{j-1}  (j+k+1) \binom{2j}{j+k+1} \binom{k-1}{i-2} = \sum_{l=i}^j  (j+l) \binom{2j}{j+l} \binom{l-2}{i-2}. \qedhere
\]
\end{proof}

\small

\bibliographystyle{plain}
\bibliography{../../../../bibliography}

\end{document}